\def\marker{\>\hbox{${\vcenter{\vbox{
    \hrule height 0.4pt\hbox{\vrule width 0.4pt height 6pt
    \kern6pt\vrule width 0.4pt}\hrule height 0.4pt}}}$}\>}

\documentclass[12pt]{article}

\usepackage[left=2cm,right=2cm,top=3cm,bottom=3cm]{geometry}

\usepackage{amsmath, amssymb, latexsym, amsthm}
\usepackage{fullpage}

\usepackage{tikz}
\usetikzlibrary{shapes}

\newtheorem{theorem}{Theorem} 
\newtheorem{theorem*}{Theorem} 

\newtheorem{lemma}[theorem]{Lemma}
\newtheorem{conjecture}{Conjecture}

\theoremstyle{definition}

\newtheorem{question}{Question}

\theoremstyle{remark}

{\end{enumerate}}

\newcommand{\RR}{\mathbb{R}}

\newcommand{\CL}[1]{\left\lceil #1 \right\rceil}

\newcommand{\FL}[1]{\left\lfloor #1 \right\rfloor}

\newcommand{\Aut}{\rm Aut}

\usepackage[hang,flushmargin]{footmisc}

\title{List-Distinguishing Cartesian Products of Cliques}
\author{Michael Ferrara\footnotemark[1], Zoltan F\"uredi\footnotemark[2], Sogol Jahanbekam\footnotemark[3],  and Paul S.\ Wenger\footnotemark[3]}

\begin{document}

\maketitle

\begin{abstract}
The {\it distinguishing number} of a graph $G$, denoted $D(G)$, is the minimum number of colors needed to produce a coloring of the vertices of $G$ so that every nontrivial isomorphism interchanges vertices of different colors.
A {\it list assignment} $L$ on a graph $G$ is a function that assigns each vertex of $G$ a set of colors.
An $L$-coloring  of $G$ is a coloring in which each vertex is colored with a color from $L(v)$.
The {\it list distinguishing number} of $G$, denoted $D_{\ell}(G)$ is the minimum $k$ such that every list assignment $L$ that assigns a list of size at least $k$ to every vertex permits a distinguishing $L$-coloring.
In this paper, we prove that when and $n$ is large enough, the distinguishing and list-distinguishing numbers of $K_n\Box K_m$ agree for almost all $m>n$, and otherwise differ by at most one.  As a part of our proof, we give (to our knowledge) the first application of the Combinatorial Nullstellensatz to the graph distinguishing problem and also prove an inequality for the binomial distribution that may be of independent interest.\\  

\end{abstract}

\renewcommand{\thefootnote}{\fnsymbol{footnote}}
\footnotetext[1]{
Department of Mathematical and Statistical Sciences, University of Colorado Denver, Denver, CO; {\tt michael.ferrara@ucdenver.edu.}\\{Research was supported in part by Simons Foundation Collaboration Grants \#206692 and \#426971.}
}
\footnotetext[2]{
Alfr\' ed R\' enyi Institute of Mathematics, Budapest, Hungary; {\tt furedi@renyi.hu.}\\
{Research was supported in part by grant K116769
from the National Research, Development and Innovation Office NKFIH, and
by the Simons Foundation Collaboration Grant \#317487.}
}
\footnotetext[3]{
School of Mathematical Sciences, Rochester Institute of Technology, Rochester, NY;
{\tt sxjsma@rit.edu, pswsma@rit.edu}.}
\renewcommand{\thefootnote}{\arabic{footnote}}

\baselineskip18pt

\section{Introduction}

Given a graph $G$, a $k$-coloring $\phi:X\to\{1,\dots,k\}$ of $G$ is \textit{distinguishing} if the only automorphism of $G$ that fixes $\phi$ is the identity automorphism.  The minimum $k$ for which $G$ has a distinguishing $k$-coloring is called the \textit{distinguishing number} of $G$, and is denoted $D(G)$.  The distinguishing number of a graph $G$ is, in some sense, a measure of the resilience of $\Aut(G)$, in that a distinguishing coloring serves to ``break" all of the symmetries of $G$.  

Inspired by a problem of Rubin \cite{Rdist79}, which creatively asked for the distinguishing number of the cycle $C_n$, Albertson and Collins initiated the study of distinguishing numbers in \cite{AC}.  Since that initial work, the distinguishing number of a graph has been determined for a number of graph classes, including planar graphs \cite{ACD}, Cartesian powers \cite{albertson,KZ}, forests \cite{cheng06} and interval graphs \cite{cheng09}.  The distinguishing problem has also been studied extensively for infinite graphs (see \cite{LeInf16} for a recent example). 

In 2011, Ferrara, Flesch, and Gethner~\cite{FFG} first considered the natural extension of graph distinguishing to list colorings.
A {\it list assignment} $L$ on a graph $G$ is a function that assigns each vertex of $G$ a set of colors.
An {\it $L$-coloring} of $G$ is a coloring in which each vertex is colored with a color from $L(v)$.
List coloring was first introduced in the setting of proper colorings by Erd\H os, Rubin, and Taylor~\cite{ERT}, and has been studied extensively across numerous settings.  The {\it list-distinguishing number} of $G$, denoted $D_{\ell}(G)$ is the minimum $k$ such that every list assignment $L$ that assigns a list of size at least $k$ to every vertex permits a distinguishing $L$-coloring.

In~\cite{FFG}, Ferrara, Flesch, and Gethner posed the following question.
\begin{question}\label{q:FFG}
Is there a graph $G$ for which $D(G)\neq D_{\ell}(G)$?
\end{question}
Question~\ref{q:FFG} is unanswered, but in subsequent years, there has been an accumulation of evidence suggesting that the negative answer to Question~\ref{q:FFG} is correct.  This includes proofs that $D(G)=D_{\ell}(G)$ when $G$ has a dihedral automorphism group \cite{FFG}, is a forest \cite{FGHSW}, and when $G$ is an interval graph \cite{IW}.  

We note that the problem of how one can distinguish the vertices of a graph, be it through coloring or the identification of a ``special" set of vertices, has been broadly studied.  As a significant recent example, Babai's proof of the existence of a near polynomial-time algorithm for the graph isomorphism problem \cite{babaiarXiv} relies on the idea of a \textit{distinguishing set} of vertices, which is a set $S$ of vertices such that every vertex not in the set has a unique subset of $S$ in its neighborhood \cite{babai80}.  Further examples include identifying codes \cite{KCL}, many variants of vertex-distinguishing edge-colorings \cite{BS}, and a number of other concepts throughout the literature.

\section{Results}

In this paper, we study the list-distinguishing number of Cartesian products of complete graphs.  The distinguishing number for Cartesian products of complete graphs was determined independently by Imrich, Jerebic, and Klav\v zar~\cite{IJK} and also Fisher and Isaak~\cite{FI}.
Interestingly, Fisher and Isaak's result was not phrased in terms of graph distinguishing, but rather in the setting of a particular class of edge colorings of the complete bipartite graph.

\begin{theorem}[\cite{IJK}]\label{thm:distknkm}
Let $n$, $m$, $k$ be positive integers with $k\ge 2$.
If $(k -1)^n < m \le k^n$, then
$$D(K_n\Box  K_m) = \begin{cases} k& \text{if } m \le k^n - \CL{\log_k n}- 1;\\
k + 1& \text{if } m \ge k^n- \CL{\log_k n} + 1.\end{cases}$$
If $m = k^n - \CL{\log_k n}$, then $D(K_n \Box K_m)$ is either $k$ or $k + 1$ and can be computed recursively in $O(log^*(m))$ time.
\end{theorem}

Our main result is the following.

\begin{theorem}\label{thm:listdistknkm}
Let $n$, $m$, and $k$ be positive integers.
If $n< m\le k^n(1-\frac{\log_{1.09}n}{\sqrt{2n}})$, then
$$D_{\ell}(K_n\Box K_m)\le k.$$
\end{theorem}

Comparing Theorems~\ref{thm:distknkm} and~\ref{thm:listdistknkm} in light of Question~\ref{q:FFG}, we see that Theorem~\ref{thm:listdistknkm} shows that for $n$ sufficiently large, the distinguishing and list-distinguishing numbers of $K_n\Box K_m$ agree for almost all $m$, and otherwise differ by at most one.

Throughout the paper we treat the vertices of $K_n\Box K_m$ as the points on an integer lattice with $n$ rows and $m$ columns.
Each copy of $K_m$ in the product will correspond to a row of the lattice, and each copy of $K_n$ in the product will correspond to a column of the lattice (see Figure~\ref{fig:knkm}).
The vertices will be labeled by $v_{i,j}$ with $1\le i\le n$ and $1\le j\le m$, with vertex $v_{1,1}$ in the top left corner and vertex $v_{n,m}$ in the bottom right corner (like the entries of an $n\times m$ matrix).

\begin{figure}
\centering
\begin{tikzpicture}

\draw[fill] (0,0) circle (2pt);
\draw[fill] (1.5,0) circle (2pt);
\draw[fill] (3,0) circle (2pt);
\draw[fill] (4.5,0) circle (2pt);
\draw[fill] (7,0) circle (2pt);

\draw[fill] (0,2.5) circle (2pt);
\draw[fill] (1.5,2.5) circle (2pt);
\draw[fill] (3,2.5) circle (2pt);
\draw[fill] (4.5,2.5) circle (2pt);
\draw[fill] (7,2.5) circle (2pt);
\draw[fill] (0,4) circle (2pt);
\draw[fill] (1.5,4) circle (2pt);
\draw[fill] (3,4) circle (2pt);
\draw[fill] (4.5,4) circle (2pt);
\draw[fill] (7,4) circle (2pt);
\draw[fill] (0,5.5) circle (2pt);
\draw[fill] (1.5,5.5) circle (2pt);
\draw[fill] (3,5.5) circle (2pt);
\draw[fill] (4.5,5.5) circle (2pt);
\draw[fill] (7,5.5) circle (2pt);

\node at (0,5.2) {$v_{1,1}$};
\node at (0,3.7) {$v_{2,1}$};
\node at (1.5,5.2) {$v_{1,2}$};
\node at (1.5,3.7) {$v_{2,2}$};
\node at (0,-.3) {$v_{n,1}$};
\node at (7,5.2) {$v_{1,m}$};
\node at (7,-.3) {$v_{n,m}$};

\node at (0,1.25) {\vdots};
\node at (1.5,1.25) {\vdots};
\node at (3,1.25) {\vdots};
\node at (4.5,1.25) {\vdots};
\node at (7,1.25) {\vdots};

\node at (5.75,0) {\dots};
\node at (5.75,2.5) {\dots};
\node at (5.75,4) {\dots};
\node at (5.75,1.25) {$\ddots$};
\node at (5.75,5.5) {\dots};

\draw[line width=2] (-.5,-.5) -- (7.5,-.5) to [bend right =90] (7.5,.5) -- (-.5,.5) to [bend right=90] (-.5,-.5);
\draw[shift={(0,2.5)}, line width=2] (-.5,-.5) -- (7.5,-.5) to [bend right =90] (7.5,.5) -- (-.5,.5) to [bend right=90] (-.5,-.5);
\draw[shift={(0,4)}, line width=2] (-.5,-.5) -- (7.5,-.5) to [bend right =90] (7.5,.5) -- (-.5,.5) to [bend right=90] (-.5,-.5);
\draw[shift={(0,5.5)}, line width=2] (-.5,-.5) -- (7.5,-.5) to [bend right =90] (7.5,.5) -- (-.5,.5) to [bend right=90] (-.5,-.5);

\draw[line width=2] (-.5,-.5) to [bend right=90] (.5,-.5) to (.5,6) to [bend right=90] (-.5, 6) to (-.5,-.5);
\draw[shift={(1.5,0)}, line width=2] (-.5,-.5) to [bend right=90] (.5,-.5) to (.5,6) to [bend right=90] (-.5, 6) to (-.5,-.5);
\draw[shift={(3,0)}, line width=2] (-.5,-.5) to [bend right=90] (.5,-.5) to (.5,6) to [bend right=90] (-.5, 6) to (-.5,-.5);
\draw[shift={(4.5,0)}, line width=2] (-.5,-.5) to [bend right=90] (.5,-.5) to (.5,6) to [bend right=90] (-.5, 6) to (-.5,-.5);
\draw[shift={(7,0)}, line width=2] (-.5,-.5) to [bend right=90] (.5,-.5) to (.5,6) to [bend right=90] (-.5, 6) to (-.5,-.5);

\node at (-1.25,0) {$K_{m}$};
\node at (-1.25,2.5) {$K_{m}$};
\node at (-1.25,4) {$K_{m}$};
\node at (-1.25,5.5) {$K_{m}$};

\node at (0,-1.25) {$K_{n}$};
\node at (1.5,-1.25) {$K_{n}$};
\node at (3,-1.25) {$K_{n}$};
\node at (4.5,-1.25) {$K_{n}$};
\node at (7,-1.25) {$K_{n}$};

\end{tikzpicture}

\caption{$K_n\Box K_m$.}\label{fig:knkm}
\end{figure}

To prove Theorem~\ref{thm:listdistknkm}, we begin by determining the list distinguishing number of $K_n\Box K_{n+1}$.
We prove this using the Combinatorial Nullstellensatz.

\begin{theorem}[Alon~\cite{Alon}]
Let $F$ be an arbitrary field, and let $f=f(x_1,\ldots,x_n)$ be a polynomial in $F[x_1,\ldots,x_n]$.
Suppose that the degree of $f$ is $\sum_{i=1}^n t_i$ where each $t_i$ is a nonnegative integer, and suppose that the coefficient of $\prod_{i=1}^n x_i^{t_i}$ in $f$ is nonzero.
If $S_1,\ldots,S_n$ are subsets of $F$ with $|S_i|>t_i$, then there are $s_1\in S_1, s_2\in S_2,\ldots,s_n\in S_n$ so that
$$f(s_1,\ldots,s_n)\neq 0.$$
\end{theorem}

We adopt the notation $f\left[\prod_{i=1}^n x_i^{t_i}\right]$ to denote the coefficient of $\prod_{i=1}^n x_i^{t_i}$ in $f$.

\begin{lemma}\label{lem:nn+1}
For $n\ge 1$, $D_{\ell}(K_{n}\Box K_{n+1})=2$.
\end{lemma}

\begin{proof}
By Theorem~\ref{thm:distknkm}, we know that $D(K_n\Box K_{n+1})=2$.
Since $D_\ell(G)\ge D(G)$ for all graphs, it follows that $D_\ell(K_n\Box K_{n+1})\ge2$.
Let $L$ be a list assignment to $V(K_n\Box K_{n+1})$ in which every list has two distinct elements, and assume that those elements are from $\RR$.
It remains to show that there is a distinguishing $L$-coloring of $K_n\Box K_{n+1}$.

To apply the Combinatorial Nullstellensatz, we will create a polynomial $F$ of degree $n^2$ in $n(n+1)$ variables such that nonzero valuations of $F$ correspond to distinguishing $L$-colorings of $K_n\Box K_{n+1}$.
The variables of the polynomial are $\{x_{i,j}\ |\ i\in[n], j\in[n+1]\}$, and the value of $x_{i,j}$ will be taken from $L(v_{i,j})$.

To build the polynomial $F$, we first define two families of polynomials whose product is $F$.
The first family is used to differentiate the columns of the graph.
The second family will then be used to differentiate the rows of the graph.

For $1\le i< j\le n+1$, define
\begin{align*}
C_{i,j}&=\sum_{k=1}^n (x_{k,j})-\sum_{k=1}^n(x_{k,i}),
\end{align*}
and set $C=\prod_{1\le i<j\le n+1}C_{i,j}$.
For $1\le h<l\le n$, define
\begin{align*}
R_{h,l}&= x_{l,h}-x_{h,h}
\end{align*}
and set $R=\prod_{1\le h<l\le n}R_{h,l}$.
Finally, define
$$F=C\cdot R.$$

We claim that a nonzero valuation of $F$ corresponds to a distinguishing coloring of $K_n\Box K_{n+1}$.
In such a valuation, $C_{i,j}\neq0$ for all $1\le i<j\le n+1$, and $R_{h,l}\neq 0$ for $1\le h<l\le n$.
If $C_{i,j}\neq 0$, then the multiset of the colors used in columns $i$ and $j$ are distinct, and hence the columns are uniquely identified by their multiset of colors.
If $R_{h,l}\neq 0$, then rows $h$ and $l$ differ in column $h$.
Working though these polynomials in order, $R_{1,2},\ldots,R_{1,n}$ use column 1 to distinguish row 1 from all other rows.
After row $h$ has been distinguished from all rows with lower indices, the polynomials $R_{h,l}$ for $l>h$ use column $h$ to distinguish row $h$ from all rows with higher indices.
Thus when $F\neq 0$ each row and column of $K_n\Box K_{n+1}$ is identifiable by its coloring, and hence the coloring is distinguishing.

Now we apply the Combinatorial Nullstellensatz to prove that $F$ has a nonzero valuation.
We have that
$$\textrm{deg}(F)=\textrm{deg}(C)+\textrm{deg}(R) = \binom{n+1}{2}+\binom{n}{2}=n^2.$$
We show that the monomial $\prod_{i\neq j}x_{i,j}$ has a nonzero coefficient in $F$.
Split the variables of the monomial into two sets: $x_{i,j}$ is {\it above the diagonal} if $i<j$, and $x_{i,j}$ is {\it below the diagonal} if $i>j$.
There are $\frac{(n+1)(n)}{2}$ variables that are above the diagonal, and these variables only occur in the polynomials $C_{i,j}$.
Since there are $\frac{(n+1)(n)}{2}$ polynomials $C_{i,j}$, it follows $\prod C_{i,j}$ must contribute the term $\prod_{i<j} x_{i,j}$ to the term $ \prod_{i\neq j}x_{i,j}$ in $F$.
It follows that $\prod R_{i,j}$ must contribute the term $\prod_{i>j} x_{i,j}$ to the term $ \prod_{i\neq j}x_{i,j}$ in $F$.
Thus
$$F\left[\prod_{i\neq j}x_{i,j}\right] = C\left[\prod_{i< j}x_{i,j}\right]\cdot R\left[\prod_{i> j}x_{i,j}\right].$$

We argue by induction on the width of our matrix that
$$\left(\prod_{1\le i<j\le n+1}C_{i,j}\right)\left[\prod_{1\le i< j\le n+1}x_{i,j}\right]=\prod_{r=1}^nr!.$$
Our induction argument will go from $k=1$ to $k=n$.
If $k=1$, then $\prod_{1\le i<j\le k+1}C_{i,j}=(x_{1,2}+\ldots+x_{n,2})-(x_{1,1}+\ldots+x_{n,1})$, and
$$\left(\prod_{1\le i<j\le 2}C_{i,j}\right)\left[\prod_{1\le i< j\le 2}x_{i,j}\right]=1=\prod_{r=1}^1r!.$$
For $2\le k \le n$, assume that
$$\left(\prod_{1\le i<j\le k}C_{i,j}\right)\left[\prod_{1\le i< j\le k}x_{i,j}\right]=\prod_{r=1}^{k-1}r!.$$
There are $k$ variables above the diagonal in column $k+1$, and they only appear in the $k$ polynomials of the form $C_{l,k+1}$ for $l\in[k]$.
Furthermore, they each appear once in each of those polynomials, and always with the coefficient $+1$.
Thus $\left(\prod_{l=1}^{k}C_{l,k+1}\right)\left[\prod_{i=1}^{k} x_{i,k+1}\right]=k!$.
Therefore,
\begin{align*}
\left(\prod_{1\le i<j\le k+1}C_{i,j}\right)\left[\prod_{1\le i< j\le k+1}x_{i,j}\right]&=\left(\prod_{1\le i<j\le k}C_{i,j}\right)\left[\prod_{1\le i< j\le k}x_{i,j}\right]\cdot \left(\prod_{l=1}^{k}C_{l,k+1}\right)\left[\prod_{i=1}^{k} x_{i,k+1}\right]\\
&=\left(\prod_{r=1}^{k-1}r!\right)\cdot k!\\
&=\prod_{r=1}^{k}r!.
\end{align*}
Thus, when $k=n$, we have that
$$\left(\prod_{1\le i<j\le n+1}C_{i,j}\right)\left[\prod_{1\le i< j\le n+1}x_{i,j}\right]=\prod_{r=1}^nr!.$$

It remains to determine $R\left[\prod_{i> j}x_{i,j}\right]$.
Each variable below the diagonal appears in exactly one polynomial $R_{h,l}$, and they all have a coefficient of $+1$.
Since they appear in distinct polynomials, it follows that
$$R\left[\prod_{i> j}x_{i,j}\right]=1.$$

We conclude that
$$F\left[ \prod_{i\neq j}x_{i,j} \right] = C\left[\prod_{i< j}x_{i,j}\right]\cdot R\left[\prod_{i> j}x_{i,j}\right]= \prod_{r=1}^n r!.$$
Since the coefficient of $\prod_{i\neq j}x_{i,j}$ in $F$ is nonzero, each variable appears with power $1$ or $0$ in the monomial, and each list has size $2$, we conclude from the Combinatorial Nullstellensatz that there is a distinguishing $L$-coloring of $K_n\Box K_{n+1}$.
\end{proof}

A closer analysis of Lemma~\ref{lem:nn+1} actually shows that any precoloring of a set of elements of $K_n\Box K_{n+1}$ that contains at most one element of each row and column can be extended to a distinguishing $L$-coloring.
We refer the interested reader to~\cite{FGHSW2} for further results about precoloring extensions in the context of graph distinguishing.

Given Lemma~\ref{lem:nn+1}, when $n$ is large enough we are able to produce distinguishing colorings of $K_n\Box K_m$ when $m$ is exponential in $n$, albeit with a base smaller than $2$.
We first include a small technical lemma that works for all values of $n$.

\begin{lemma}\label{lem:bigcoeff}
Let $S(x_{1,1},x_{1,2},x_{2,1},x_{2,2},\dots, x_{n,1},x_{n,2})=\prod_{i=1}^n(x_{i,1}+x_{i,2})$.  For any assignment of the $x_{i,j}$ to formal variables from $\{c_1,\ldots,c_r\}$ such that $x_{i,1}\ne x_{i,2}$ for all $i\in[n]$, the coefficient of each monomial $\prod_{i=1}^rc_i^{\alpha_i}$ in $S$ is at most ${{n}\choose{\lceil\frac{n}{2}\rceil}}$.
\end{lemma}

\begin{proof}
We apply induction on $n$ to prove the assertion.
When $n=1$, every term has coefficient $1$, as desired.

If $\{x_{i,1},x_{i,2}\}=\{x_{j,1},x_{j,2}\}$ for all $i$ and $j$, then $S(c_1,\ldots,c_r)=\sum_{k=1}^n\binom{n}{k}x_{1,1}^kx_{1,2}^{n-k}$, and the result holds.
Therefore we assume there are values $i$ and $j$ such that $\{x_{i,1}, x_{i,2}\}\neq \{x_{j,1},x_{j,2}\}$.
Assume without loss of generality that $c_1$ occurs exactly in the first $\beta$ terms, where $\beta\in[n-1]$.
Write $S=S_1S_2$, where $S_1(x_{1,2},\ldots,x_{{\beta},2})=\prod_{i=1}^{\beta}(c_1+{x_{i,2}})$, $S_2(x_{\beta+1,1}, x_{\beta+1,2},\ldots ,x_{n,1},x_{n,2})=\prod_{j=\beta+1}^{n}(x_{j,1}+x_{j,2})$, and $x_{i,2},x_{j,1},x_{j,2}\in \{c_2,\ldots,c_r\}$.

In order to create $\prod_{i=1}^rc_i^{\alpha_i}$ in $S$, we must choose $c_1$ from $\alpha_1$ terms in $S_1$.
There are ${\beta}\choose{\alpha_1}$ ways we can create $c_1^{\alpha_1}$ in $S_1$.
Fix one term in $S_1$ in which the power of $c_1$ is $\alpha_1$.
Suppose that this term is of the form $c_1^{\alpha_1}\prod_{i=2}^rc_i^{\gamma_i}$.
This term can be used to create $\prod_{i=1}^rc_i^{\alpha_i}$ in $S$ only if $\prod_{i=2}^rc_i^{\alpha_i-\gamma_i}$ has a nonzero coefficient in $S_2$.
By the induction hypothesis, $\prod_{i=2}^rc_i^{\alpha_i-\gamma_i}$ has coefficient at most ${n-\beta}\choose{\CL{\frac{n-\beta}{2}}}$ in $S_2$.
Therefore $\prod_{i=1}^rc_i^{\alpha_i}$ has coefficient at most ${{\beta}\choose{\alpha_1}}{{n-\beta}\choose{\CL{\frac{n-\beta}{2}}}}$ which is at most $\binom{n}{\CL{\frac{n}{2}}}$, as desired.
\end{proof}

\begin{lemma}\label{lem:cubic}
For $n\ge 2222$ and $n<m\le \CL{1.09^n}+ n+1$,
$$D_{\ell}(K_{n}\Box K_{m})\le 2.$$
\end{lemma}

\begin{proof}
Let $G=K_n\Box K_m$.
Let $L$ be a list assignment on $V(G)$ in which all lists have size $2$.
The proof proceeds in two steps.
In the first step, we select a set of columns which we will color so that no nontrivial row transposition of $G$ is color-preserving.
In the second step we extend this coloring to all of the other columns so that no nontrivial column transposition is color-preserving.

We say that a column is {\it list-uniform} if every vertex in the column has the same list of colors.
Let $A'$ be the set of list-uniform columns.
If $|A'|\ge n+1$, set $A=A'$.
If $|A'|<n+1$, let $A$ consist of $A'$ with $n+1-|A'|$ additional columns, chosen arbitrarily.
Let $V(A)$ denote the set of vertices contained in the columns in $A$.

Define the {\it color pattern} of a colored column to be the multiset of the colors used to color the vertices of the column.
Define the {\it color vector} of a colored column to be the vector of the colors used to color the vertices of the column, $\langle c(v_{1,j}),c(v_{2,j}),\ldots,c(v_{n,j})\rangle$.
We will color the columns of $G-A$ so that (1) no column in $G-A$ shares its color pattern with a column in $A$, and (2) no two vectors in $G-A$ have the same color vector.

If $|A'|\ge n+1$, then $A=A'$.
Fix a distinguishing $2$-coloring $c$ of $K_{n}\Box K_{|A|}$ with no monochromatic columns, and let the colors be $0$ and $1$.
Theorem~\ref{thm:distknkm} implies that such a coloring exists because $\CL{1.09^n}+n+1\le 2^n\CL{\log_2 n}-3$.
Assume that the colors of the graph have been ordered.
Color $G[V(A)]$ using $c$ so that in column $j$, the vertices colored $0$ are assigned the minimum color in the list of column $j$ and the vertices colored $1$ are assigned the maximum color in the list of column $j$.
This is a distinguishing coloring of $G[V(A)]$.
Now, all remaining columns contain at least three colors.
Greedily color these columns so that (1) there are at least three different colors on their vertices and (2) so that their color vectors are all different from each other.
It follows that there are at least $2^{n-3}$ colorings of each column, and because $n\ge 4$, we have that $2^{n-3}\ge \CL{1.09^n}$.
The resulting coloring of $G$ is distinguishing because the coloring of each column in $A$ uses two colors and distinguishes the rows of $G$, while the remaining columns are all distinguished because they have distinct color vectors with at least three colors.

For the rest of the proof, assume that $|A'|<n+1$.
In this case, we first obtain a distinguishing coloring of $G[V(A)]$ by using Lemma~\ref{lem:nn+1}.
For $C_j\notin A$, let $L(C_j)=\bigcup_{i=1}^nL(v_{i,j})$, and note that $|L(C_j)|\ge 3$.
Therefore there is a color $c_i$ in this union that appears in at most $2n/3$ lists in the column.
We will show that each column in $G-A$ has at least $1.09^n$ possible color vectors that do not share their color pattern with any column of $A$.
We will proceed with two cases.

{\bf Case 1:} Two colors appear in at least $\frac 78 n-n^{2/3}$ terms.
It follows that there are at least $\frac 34 n-2n^{2/3}$ lists that have the same list of two colors.
Let $L(C_j)=\{c_1,\ldots,c_\ell\}$ and let $\{c_1,c_2\}$ be the list on at least $\frac 34 n-2n^{2/3}$ vertices.

Since the column is not list-uniform, there are at least two distinct color patterns on $C_j$ for the colors $\{c_3,\ldots,c_\ell\}$.
Thus we may color the vertices in $C_j$ whose lists are not $\{c_1,c_2\}$ so that at most $(n+1)/2$ columns in $A$ match that color pattern on $\{c_3,\ldots,c_\ell\}$.

The extensions of the partial coloring of $C_j$ have at least $\frac 34 n-2n^{2/3}+1$ distinct color patterns, depending on the number of vertices with list $\{c_1,c_2\}$ that receive color $c_1$.
Therefore, there is an $\alpha\in\left\{\FL{\frac n8-n^{2/3}}-1,\FL{\frac n8-n^{2/3}},\ldots,\CL{\frac{5n}{8}-n^{2/3}}+1\right\}$ such that none of the patterns comes from a column in $A$ with at least $\alpha$ vertices colored $c_1$.
Therefore, for $n\ge 2222$, it follows that the number of extensions of the colorings to $C_j$ that do not match the color pattern of any column in $A$ is bounded by 
$$\binom{3n/4-2n^{2/3}}{\FL{n/8-n^{2/3}}-1}\ge \left(\frac{3n/4-2n^{2/3}}{n/8-n^{2/3}-1}\right)^{n/8-n^{2/3}}\ge 6^{n/8-n^{2/3}}\ge 2^{n/8}\ge \CL{1.09^{n}}.$$
This finishes Case 1.

{\bf Case 2:} No two colors appear in at least $\frac 78 n-n^{2/3}$ terms.
Thus there is at most one color that is used in at least $\frac 78 n-n^{2/3}$ terms.
Let $L(C_j)=\{c_1,\ldots,c_\ell\}$ such that
\begin{enumerate}
\item if a color appears at least $\frac 78 n-n^{2/3}$ terms, then it is $c_\ell$;
\item colors $\{c_1,\ldots,c_{\ell-1}\}$ are ordered so that the function $k(h)$, which denotes the number of vertices in $C_j$ whose list contains $c_h$ but does not contain $c_g$ for any $g<h$, is decreasing.
\end{enumerate}
Thus there is a set $D=\{c_1,\ldots,c_{\ell'}\}$ such that
$$n^{2/3}< \sum_{i=1}^{\ell'} k(i) \le \frac 78n.$$

We will color the vertices in $C_j$ with color $\{c_1,\ldots,c_h\}$ in their lists so that after the coloring there are at most
$$\frac{n+1}{\prod_{i=1}^h(k(i)+1)}$$
columns in $A$ whose color pattern in colors $\{c_1,\ldots,c_h\}$ matches the color pattern of $C_j$ in those colors.
This statement is trivially true before any vertices in $C_j$ have been colored.
Prior to coloring the vertices with color $c_h$ in their lists, there are at most
$$\frac{n+1}{\prod_{i=1}^{h-1}(k(i)+1)}$$
columns in $A$ whose color pattern matches the color pattern of $C_j$ on the first $h-1$ colors.
When processing color $c_h$, there are $k(h)$ vertices with the color $c_h$ in their lists that have not been colored (note that some vertices with $c_h$ in their list may have been colored when earlier colors were processed).
Therefore there are $k(h)+1$ possibilities for the number of vertices with color $c_h$ in $C_j$.
By the pigeonhole principle, it follows that there is a choice for the number of vertices that will be colored $c_h$ so that there are only
$$\frac{n+1}{\prod_{i=1}^{h}(k(i)+1)}$$
columns in $A$ whose color pattern matches the color pattern of $C_j$ on the first $h$ colors.
After processing all of the colors in $D$, the number of columns in $A$ whose color pattern matches the color pattern of $C_j$ on the colors in $D$ is bounded by 
$$\frac{n+1}{\prod_{i=1}^{\ell'}(k(i)+1)} < \frac{n+1}{n^{2/3}}<n^{1/3}.$$

For simplicity, let $k=\sum_{i=1}^{\ell'}k(i)$.
At this point, there are $n-k$ uncolored vertices in $C_j$.
Therefore, there are $2^{n-k}$ extensions of the coloring to all of $C_j$.
By Lemma~\ref{lem:bigcoeff}, each column in $A$ whose color pattern matches the color pattern of $C_j$ on the colors in $D$ can match the color pattern of at most $\binom{n-k}{\CL{(n-k)/2}}$ of the colorings of $C_j$.
When $n-k$ is even, we will use the approximation $\binom{n-k}{(n-k)/2}\le \frac{2^{n-k}}{\sqrt{ \frac32(n-k)+1}}$.
When $n-k$ is odd, this approximation gives us $\binom{n-k}{\CL{(n-k)/2}}\le \frac{2^{n-k}}{\sqrt{ \frac32(n-k)-\frac 12}} \frac{n-k}{n-k+1}$.
Thus $\binom{n-k}{\CL{(n-k)/2}}\le \frac{2^{n-k}}{\sqrt{ \frac32(n-k)+1}}$ is true for all values of $n-k$.
Therefore, because $n\ge 2194$, the number of colorings of $C_j$ whose color pattern matches the color pattern of no column in $A$ is bounded by
\begin{align*}
2^{n-k}-n^{1/3}\binom{n-k}{\CL{(n-k)/2}}&\ge 2^{n-k}-n^{1/3}\frac{2^{n-k}}{\sqrt{ \frac32(n-k)+1}}\\
& \ge 2^{n/8}\left(1-\frac{n^{1/3}}{\sqrt{ \frac3{16} n+1}}\right)\\
&\ge \CL{1.09^n}.
\end{align*}
This finishes Case 2.

In both cases, each column in $G-A$ has at least $\CL{1.09^n}$ colorings that do not match the color pattern of any column in $A$.
Greedily choose such colorings for the columns of $G-A$ so that no two of these columns have the same color vector.

We claim that the coloring we have generated is a distinguishing coloring of $G$.
By construction, the columns in $A$ are distinguishable from the columns not in $A$ since the color patterns of the columns in $A$ are not repeated outside of $A$.
Once the columns of $A$ have been identified, the coloring of $A$ is distinguishing on the induced subgraph of $A$.
Therefore the coloring of the columns of $A$ uniquely identifies the vertices in $A$ and hence distinguished the rows of $G$.
Finally, once the rows have been distinguished, each column of $G-A$ is distinguishable, since they all have distinct color vectors.
Thus the coloring is a distinguishing coloring.
\end{proof}

We now proceed to the proof of Theorem~\ref{thm:listdistknkm}.
We first give two lemmas that provide a bound on the coefficients that we will see in the generating functions that allow us to color our graph.

 Although the following inequality seems to be quite natural, several experts in probability theory were not aware of it, and a classical book of inequalities~\cite{BB} did not include it.  Hence, it may be of independent interest.  

\begin{lemma}\label{lem:binomdistbound}
If $a$ and $n$ are positive integers with $a<n$ and $0<p<1$ is a real number, then $${n\choose a}p^a(1-p)^{n-a}<\frac{C}{\sqrt{np(1-p)}},$$ where $C=\left(\frac{3}{2e}\right)^{\frac{3}{2}}.$
\end{lemma}

While the proof of Lemma \ref{lem:binomdistbound} relies only on elementary calculus, it is somewhat long and we feel that a full proof would detract from the proof of Theorem~\ref{thm:listdistknkm}, so have elected to present it in full in Appendix~\ref{sec:App}.

\begin{lemma}\label{lem:bigcoeffk}
Let $k$ be an integer that is at least three and let
$$S(x_{1,1},\ldots,x_{1,k},x_{2,1},\ldots,x_{2,k},\ldots, x_{n,1},\ldots,x_{n,k})=\prod_{i=1}^n(x_{i,1}+\ldots+x_{i,k}).$$
For any assignment of the $x_{i,j}$ to formal variables from $\{c_1,\ldots,c_r\}$ such that $x_{i,j}$ are all distinct for all $j\in k$, the coefficients of each monomial $\prod_{i=1}^rc_i^{\alpha_i}$ in $S$ are all at most $\frac{Ck^{n+1}}{\sqrt[4]{n}}$, where $C=\left(\frac{3}{2e}\right)^{\frac{3}{2}}$.
\end{lemma}

\begin{proof}
Let $P_i=x_{i,1}+x_{i,2}+\ldots+x_{i,k}$.
If there is a pair of variables $c_j$ and $c_{j'}$ such that no $P_i$ contains both of them, then replace all appearances of $c_{j'}$ with $c_{j}$ in $S$ to obtain a new polynomial $S'$.
Since the coefficient of $c_j^{\alpha_j+\alpha_{j'}}\prod_{i\neq j,j'}c_i^{\alpha_i}$ in $S'$ is at least the coefficient of $c_1^{\alpha_1}\ldots c_r^{\alpha_r}$ in $S$, we may make the assumption that every pair of variables $c_j$ and $c_{j'}$ appear together in at least one term.
Therefore we have ${{r}\choose{2}}\leq n{{k}\choose{2}}$, which implies that
$$r\leq \frac{1+\sqrt{1+4nk(k-1)}}{2}<\sqrt{2nk^2}.$$

For each $i\in [r]$, let $\beta_i$ be the number of terms that contain $c_i$.
Without loss of generality assume that $\beta_1\geq \beta_i$ for each $i$.
By the choice of $\beta_1$, we have that $\beta_1\ge \CL{nk/r}$.
Because $r<\sqrt{2nk^2}$, we have 
\begin{equation}\label{eq1}
\beta_1> \frac{nk}{\sqrt{2nk^2}}=\sqrt{\frac{n}{2}}.\end{equation}

The coefficient $c^*$ of $c_1^{\alpha_1}\ldots c_r^{\alpha_r}$ in $S$ is at most ${{\beta_1}\choose{\alpha_1}}(k-1)^{\beta_1-\alpha_1}(k)^{n-\beta_1}$.
Note that if $\alpha_1=0$, then this value is at most 
\begin{equation*}
c^*\leq (k-1)^{\beta_1}k^{n-\beta_1} \le (k-1)^{\CL{nk/r}}k^{n-\CL{nk/r}}= k^n\left(1-\frac 1k\right)^{\CL{nk/r}}.
 \end{equation*}
Using our upper bound~\eqref{eq1} on $r$, and the inequality $(1-x)^y\leq e^{-xy}$ (true for all $0\leq x<1$, $y\geq 0$) we have
\begin{equation*}
c^*  \le k^n \left(1-\frac 1k\right)^{\sqrt{n/2}} \le \frac{k^n}{e^{\sqrt{n/2k^2}}} <\frac{Ck^{n+1}}{\sqrt[4]{n}}.
\end{equation*}
If $\alpha_1=\beta_1$, then the coefficient is at most $k^{n-\beta_1}$, which is also bounded by $\frac{Ck^{n+1}}{\sqrt[4]{n}}$ by the argument above.
Thus we may assume that $0< \alpha_1<\beta_1$.

By Lemma~\ref{lem:binomdistbound}, ${n\choose a}p^a(1-p)^{n-a}<\frac{C}{\sqrt{np(1-p)}}$ when $a$ and $n$ are positive integers with $a<n$ and $p\in(0,1)$.
Therefore, letting $a=\alpha_1$, $n=\beta_1$, and $p=\frac 1k$, we have
$$\frac{c^*}{k^n}\leq {{\beta_1}\choose{\alpha_1}}\frac{1}{k^{\alpha_1}} \left( \frac{k-1}{k}\right)^{\beta_1-\alpha_1}<
   \frac{Ck}{\sqrt{\beta_1 (k-1)}}.$$
Apply~\eqref{eq1} again, we have that the maximum coefficient in $S$ satisfies
\begin{equation*}
c^*< \frac{C\sqrt{2} \cdot k^{n+1}}{\sqrt{k-1}\sqrt[4]{n}} < \frac{C\cdot k^{n+1}}{\sqrt[4]{n}}.\qedhere
\end{equation*}
\end{proof}

Given Lemmas~\ref{lem:bigcoeff}, \ref{lem:cubic}, and~\ref{lem:bigcoeffk}, we can now show that with lists of size $k$ we are able to find a distinguishing colorings of $K_n\Box K_m$ for $n$ sufficiently large where $m=k^n(1-o(1))$.

\begin{theorem}\label{thm:k^n}
For $n$ sufficiently large and $n<m\le  k^n(1+o(1))$,
$$D_{\ell}(K_{n}\Box K_{m})\le k.$$
\end{theorem}

\begin{proof}
Let $A$ be the set of the first $\CL{\log_{1.09}n}$ columns in $G$.
By Theorem~\ref{lem:cubic}, there is a distinguishing coloring of the graph induced by $A$.

First consider the case when $k=2$.
By Lemma~\ref{lem:bigcoeff}, each column in $A$ has a color pattern that matches the color pattern of at most $\binom{n}{\CL n2}$ of these colorings.
Therefore the number of colorings of each column in $G-A$ that does not have a color pattern of a column in $A$ is at least
\begin{align*}
k^n-\CL{\log_{1.09}n}\binom{n}{\CL{n/2}}&\ge k^n\left(1-\frac{1}{\sqrt{3n/2+1}} \right)\\
&=k^n(1-o(1)).
\end{align*}
Otherwise, $k\ge 3$ and each remaining column has $k^n$ colorings.
By Lemma~\ref{lem:bigcoeffk}, each column in $A$ has a color pattern that matches the color pattern of at most $\frac{Ck^{n+1}}{\sqrt[4]{n}}$ of these colorings.
Therefore the number of colorings of each column in $G-A$ that does not have a color pattern of a column in $A$ is at least
\begin{align*}
k^n-\CL{\log_{1.09}n}\left(\frac{Ck^{n+1}}{\sqrt[4]{n}}\right)&=k^n\left(1-\frac{Ck\CL{\log_{1.09}n} }{\sqrt[4]{n}} \right)\\
&=k^n(1-o(1)).
\end{align*}
To complete the coloring, we greedily select colorings for the columns in $G-A$ that do not have the same color pattern as the columns in $A$ so that no two of these columns have the same color vector.

As in the proof of Lemma~\ref{lem:cubic}, this coloring is distinguishing since the columns in $A$ are identifiable by their color patterns.
The coloring on $A$ is distinguishing on those columns, and hence it distinguishes the rows of $G$.
All remaining columns are distinguishable due to their distinct color vectors.
\end{proof}

We do observe here that the value of $n$ is quite large for Theorem~\ref{thm:k^n} to apply.
Since we depend on the width of the first collection of columns to be given by $\CL{\log_{1.09}n}$ to apply Lemma~\ref{lem:cubic}, we require $\CL{\log_{1.09}n}\ge 2222$.
Thus $n\ge 1.45104 \times 10^{83}$ is sufficient for the theorem.

\section{Open Questions}

Question~\ref{q:FFG} is the clearest future direction to go with this work.
We ask the question for the list coloring of grids here.
\begin{conjecture}
For all $n$ and $m$, $D_\ell(K_n\Box K_m)$ equals $D(K_n\Box K_m)$.
\end{conjecture}

In the process of proving Theorem \ref{thm:listdistknkm}, we encountered a seemingly simple conjecture that we approximated with Lemma~\ref{lem:bigcoeffk}.
We present this conjecture here.
\begin{conjecture}
Let $k$ be an integer that is at least three and let
$$S(x_{1,1},\ldots,x_{1,k},x_{2,1},\ldots,x_{2,k},\ldots, x_{n,1},\ldots,x_{n,k})=\prod_{i=1}^n(x_{i,1}+\ldots+x_{i,k}).$$
For any assignment of the $x_{i,j}$ to formal variables from $\{c_1,\ldots,c_r\}$ such that $x_{i,j}$ are all distinct for all $j\in k$, the coefficient of each monomial $\prod_{i=1}^rc_i^{\alpha_i}$ in $S$ are all at most the balanced $k$-multinomial coefficient
$$\binom{n}{\CL{n/k},\ldots,\FL{n/k}}.$$
\end{conjecture}

\appendix
\section{An inequality concerning the binomial distribution}\label{sec:App}

Here we prove  Lemma~\ref{lem:binomdistbound}.

\medskip
\noindent
{\bf Lemma~\ref{lem:binomdistbound}.}\enskip
{\sl
If $a$ and $n$ are positive integers with $a<n$, and $0<p<1$ is a real number, then $$ {n\choose a}p^a(1-p)^{n-a}<\frac{C}{\sqrt{np(1-p)}},$$ where $C=\left(3/2e\right)^{3/2}=0.409916\dots$.
}
\medskip

\begin{proof}
Define $f(n,a,p)= \sqrt{n} {n\choose a}p^{a+(1/2)}(1-p)^{n-a+(1/2)}$.
Our aim is to show that 
$$f(n,a,p)< C.$$
For $1\leq a\leq  n-1$, the function $g(p)=p^{a+(1/2)}(1-p)^{n-a+(1/2)}$ takes its maximum in the interval $[0,1]$ when 
\begin{align*}
0&=g'(p)\\
&=p^{a+(1/2)}(n-a-(1/2))(1-p)^{n-a-(1/2)}+(a+(1/2))p^{a-(1/2)}(1-p)^{n-a+(1/2)}\\
&=p^{a-(1/2)}(1-p)^{n-a-(1/2)}\left[\left(n-a+\frac{1}{2}\right)p-\left(a+\frac{1}{2}\right)(1-p)\right].
\end{align*}
Therefore
$$f(n,a,p) \leq f\left(n,a, \frac{a+(1/2)}{n+1}\right).$$
Define
$$f(n,a)=f\left(n,a, \frac{a+(1/2)}{n+1}\right)=\sqrt n \binom na \left(\frac{a+(1/2)}{n+1}\right)^{a+(1/2)}\left(1-\frac{a+(1/2)}{n+1}\right)^{n-a+(1/2)}.$$
Since $f(2,1)=\sqrt{2}/4<0.35356$ and $f(3,2)=f(3,1)= (675/4096)\sqrt{5}<0.36850$, these values are less than $C$.
From now on we may suppose that $n\geq 4$.
Also, $f(n,a)=f(n, n-a)$, so from now on we may suppose that $1\leq a\leq \frac{n}{2}$.

We now wish to show that $f(n,a)$ is decreasing in terms of $a$ so that $f(n,1)\ge f(n,a)$.
We have that for $2\le a\le n/2$,
\begin{align*}
\frac{f(n,a-1)}{f(n,a)}&=\frac{\sqrt n \binom n{a-1} \left(\frac{a-\frac 12}{n+1}\right)^{a-\frac 12}\left(1-\frac{a-\frac 12}{n+1}\right)^{n-a+\frac 32}}{\sqrt n \binom n{a} \left(\frac{a+\frac 12}{n+1}\right)^{a+\frac 12}\left(1-\frac{a+\frac 12}{n+1}\right)^{n-a+\frac 12}}\\
&=\frac{\frac{1}{n-a+1}\left(a-\frac 12\right)^{a-\frac 12}\left(n-a+\frac32 \right)^{n-a+\frac 32}}{\frac{1}{a}\left(a+\frac 12\right)^{a+\frac 12}\left(n-a+\frac12\right)^{n-a+\frac12}}\\
&=\frac{a\frac{(a-\frac 12)^{a-\frac 12}}{(a+\frac 12)^{a+\frac 12}}}{(n-a+1)\frac{(n-a+\frac 12)^{(n-a+\frac 12)}}{(n-a+\frac 32)^{(n-a+\frac 32)}}}.\\
\end{align*}
Define the sequence
$$X(a)=a\frac{{(a-(1/2))}^{a-(1/2)}} {{(a+(1/2))}^{a+(1/2)}}.$$
Thus we wish to determine if $X(a)/X(n-a+1)>1$ under the condition that $1\le a\le n/2$.
Taking the logarithm of both sides yields the following inequality that we wish to prove:
\begin{align*}
0&<\log(X(a)/X(n-a+1))\\
&=\log(X(a)) - \log(X(n-a+1)).
\end{align*}
For all real  $a>0$, define the continuous function $x(a)= \log X(a+(1/2))$, so
$$  x(a)= \log (a+(1/2)) + a \log a -(a+1) \log (a+1).  $$
We claim that $x(a)$ is strictly decreasing.

To prove that $x(a)$ is decreasing, we use the fact that if $x(a):(0, \infty)\to \mathbb R$ is a real function that is convex and $\lim_{a\to \infty} x(a)$ exists and is finite, then $x(a)$ is strictly decreasing.
The limit of $x(a)$ exists since
 $$
   \lim_{a\to \infty} x(a)= \lim_{a\to \infty} \log \frac{(a+\frac12)}{(a+1)} + \lim_{a\to \infty}  \log \left(\left(\frac{a}{a+1}\right)^a\right) =-1.
  $$
We have that 
$$x'(a)= (a+1/2)^{-1}+ (1+ \log a) -(1+ \log (a+1))$$
and therefore
\begin{align*}
  x''(a)&= \frac{-1}{(a+(1/2))^2} + \frac{1}{a}-\frac{1}{a+1}\\
  & = \frac{-1}{a^2+a+ (1/4)} + \frac{1}{a^2+a}\\
  &>0.
\end{align*}
Hence the function $x(a)$ is convex (for $a>0$) and has a finite limit, so $x(a)$ is strictly decreasing.
Therefore we have shown that $f(n,1)\ge f(n,a)$, so
$$f(n,a) \leq n^{3/2}\left(\frac{3/2}{n+1}\right)^{3/2}
       \left(\frac{n-(1/2)}{n+1}\right)^{n-(1/2)}.$$

Finally, define the function
$$f(n) = n^{3/2}\left(\frac{3/2}{n+1}\right)^{3/2}
       \left(\frac{n-(1/2)}{n+1}\right)^{n-(1/2)}.$$
For all real $n> 1$ define the function $y(n)= \log f(n)$, so
$$
  y(n)= \frac{3}{2}\log \frac32+\frac{3}{2}\log n + \left(n-\frac12\right) \log \left(n-\frac12\right) -(n+1) \log (n+1). $$
We claim that $y(n)$ is strictly increasing.

To prove that $y(n)$ is increasing, we use the fact that if $y(n):(1, \infty)\to \mathbb R$ is a real function that is concave and $\lim_{n\to \infty} y(n)$ exists and is finite, then $y(n)$ is strictly increasing.
The limit of $y(n)$ exists since
$$\lim_{n\to \infty} y(n)=\frac32\log \frac32-\frac 32.$$
We have that
$$y'(n)= (3/2)n^{-1}+ \left(1+ \log  \left(n-\frac12\right)\right) -(1+ \log (n+1)),$$
and therefore
\begin{align*}
y''(n)&= \frac{-3/2}{n^2} + \frac{1}{n-(1/2)}-\frac{1}{n+1}\\
& =  \frac{-3/2}{n^2} + \frac{3/2}{n^2+(n/2)-(1/2)}\\
&<0.
\end{align*}
Hence the function $y(n)$ is concave (for $n>1$) and has a finite limit, so $y(n)$ is strictly increasing.
Hence $f(n)$ is also increasing, and $f(n)<(3/2e)^{3/2}$.

Therefore we have
\begin{align*}
f(n,a,p)&\le f\left(n,a,\frac{a+\frac12}{n+1}\right)\\
&\le f(n,1)\\
&< \lim_{n\to\infty} f(n,1)\\
&=\left(\frac{3}{2e}\right)^{3/2}.\qedhere
\end{align*}
\end{proof}

 \bigskip
Considering the sequence of $f(n,1)$ one can see that the value of $C$ is the best possible.

It is well known that ${n \choose n/2}< 2^n /\sqrt{2\pi n}$ for all even $n\geq 2$, so one would expect that Lemma~\ref{lem:binomdistbound} should hold with
$D=1/\sqrt{2\pi}$, but it is only $0.398942\dots $, about 2.7\% smaller than $C$.
However, using the first line of the proof and the precise Sterling formula
(i.e., $n!=n^ne^{-n}\sqrt{2\pi n} \exp(1/(12n+\Theta_n))$ where $0<\Theta_n<1$),
 one can prove that if $a\to \infty, \, a\leq n/2$, then $f(n,a)= (1+o(1))/\sqrt{2\pi}$.


\begin{thebibliography}{9}
\frenchspacing

\bibitem{albertson}
M. Albertson,
Distinguishing Cartesian powers of graphs.
{\it Electron. J. Combin.} {\bf 12} (2005), Note 17, 5 pp. (electronic).


\bibitem{AC}
M. Albertson and K.L.\ Collins,
Symmetry breaking in graphs,
{\it Electron. J. Combin.} {\bf 3} (1996), no.\ 1, Research Paper 18, 17 pp. (electronic).

\bibitem{Alon}
N.\ Alon, Combinatorial Nullstellensatz, Recent trends in combinatorics, {\it Combin. Probab. Comput.} {\bf 8} (1999), no. 1-2, 7--29.

\bibitem{ACD}
V. Arvind, C. Cheng and N. Devanur,
On computing the distinguishing numbers of planar graphs and beyond: a counting approach, {\it SIAM J.\ Discrete Math.} {\bf 22} (2008), no.\ 4, 1297--1324.


\bibitem{AD}
V.\ Arvind and N.\ Devanur,
Symmetry breaking in trees and planar graphs by vertex coloring,
in {\it Proceedings of the 8th Nordic Combinatorial Conference}, Aalborg University, Aalborg, Denmark, 2004.

\bibitem{babaiarXiv} L. Babai, Graph isomorphism in quasipolynomial time,
arXiv preprint arXiv:1512.03547[cs.DS], 89 pp.

\bibitem{babai80} L. Babai, On the complexity of canonical labeling of strongly regular graphs.
\textit{SIAM J. Comput.} \textbf{9} (1980), 212--216.

\bibitem{BB}
E. Beckenbach and R.\ Bellman,
\newblock {\sl Inequalities,}
\newblock Springer-Verlag, New York, 1961. 

\bibitem{BS} A. Burris and R. Schelp, Vertex-distinguishing proper edge-colorings, \textit{J. Graph Theory} \textbf{26} (1997). 73--82.



\bibitem{cheng09}
C.\ Cheng,
On computing the distinguishing and distinguishing chromatic numbers of interval graphs and other results,
{\it Discrete Math.} {\bf 309} (2009), no.\ 16, 5169--5182.

\bibitem{cheng06}
C.\ Cheng,
On computing the distinguishing numbers of trees and forests,
{\it Electron. J. Combin.} {\bf 13} (2006), no.\ 1, Research Paper 11, 12 pp. 






\bibitem{ERT}
P. Erd\H{o}s, A. Rubin, and H. Taylor,
Choosability in graphs,
{\it Congr. Numer.} {\bf 26} (1980), 125--157.


\bibitem{FFG}
M. Ferrara, B. Flesch, and E. Gethner,
List-distinguishing colorings of graphs,
{\it Electron. J. Combin.} {\bf 18} (2011), no.\ 1, Paper 161, 17 pp.

\bibitem{FGHSW}
M. Ferrara, E. Gethner, S. Hartke, D. Stolee, and P. Wenger,
List distinguishing parameters of trees,
{\it Discrete Appl. Math.} {\bf 161} (2013), no.\ 6, 864--869.


\bibitem{FGHSW2}
M. Ferrara, E. Gethner, S. Hartke, D. Stolee, and P. Wenger,
Extending procolorings to distinguish group actions,
arXiv:1405.5558 [math.CO].


\bibitem{FI}
M. Fisher, and G. Isaak,
Distinguishing colorings of Cartesian products of complete graphs,
{\it Discrete Math.} {\bf 308} (2008), no.\ 11, 2240--2246

\bibitem{IW}
P.\ Immel and P.\ Wenger,
The list distinguishing number equals the distinguishing number for interval graphs,
{\it Discuss.\ Math.\ Graph Theory} {\bf 37} (2017), 165--174.



\bibitem{IJK}
W. Imrich, J.\ Jerebic, and S. Klav\v zar,
The distinguishing number of Cartesian products of complete graphs,
{\it European J. Combin.} {\bf 29} (2008), no.\ 4, 922--929.




\bibitem{KCL} M. Karpovsky, K. Chakrabarty and L. Levitin, On a new class of codes for identifying vertices in graphs, \textit{IEEE Trans. Inform. Theory}, \textbf{44} (1998), 599--611.





\bibitem{KZ}
S. Klav\v zar, and X. Zhu,
Cartesian powers of graphs can be distinguished by two labels.
{\it European J. Combin.} {\bf 28} (2007), no. 1, 303--310.


  
\bibitem{LeInf16} F. Lehner, Distinguishing graphs with intermediate growth,
\textit{Combinatorica}, \textbf{36} (2016), 333--347. 
  
\bibitem{Rdist79}
F. Rubin, Problem 729, \textit{J. Recreational Math.} \textbf{11} (1979), 128.


  
\end{thebibliography}
\end{document}